\documentclass[a4paper,draft]{amsproc}
\usepackage{amssymb}
\usepackage[hyphens]{url} 

\theoremstyle{plain}
 \newtheorem{theorem}{Theorem}[section]

 \newtheorem{corollary}{Corollary}[section]
\theoremstyle{definition}

\theoremstyle{remark}
 \newtheorem{remark}{Remark}[section]

 \numberwithin{equation}{section}

\renewcommand{\geq}{\geqslant}

\title[Special polynomials/ M. Goubi]{Remarks on some properties of special polynomials with exponential distribution}



\author{\bfseries Mouloud  Goubi} 
\address{Mouloud Goubi\\
Department of Mathematics \\
University of UMMTO RP. 15000\\
Tizi-ouzou, Algeria\\
Laboratoire d'Alg\`ebre et Th\'eorie des Nombres, USTHB Alger}
\email{mouloud.goubi@ummto.dz}

\begin{document}

\vspace{18mm} \setcounter{page}{1} \thispagestyle{empty}

\begin{abstract}
In this notice, we revisit the recent work \cite{Kang} of Jung Yoog
Kang and Tai Sup about special polynomials with exponential
distribution in order to state some improvements and get new proofs
for results therein.
\end{abstract}

\maketitle

\section{Introduction}
In this notice we revisit the recent work \cite{Kang} on some
properties of special polynomials with exponential distribution of
Jung Yoog Kang and Tai Sup Lee published in Commun. Korean Math.
Soc. The object of this study is the family of polynomials
$\mathfrak{E}_n\left(\lambda:x\right)$ generated by the generating
function
\[\frac{\lambda}{e^{\lambda
t}}e^{xt}=\sum_{n\geq0}\mathfrak{E}_n\left(\lambda:x\right)\frac{t^n}{n!}\]
and associated numbers
$\mathfrak{E}_n\left(\lambda\right)=\mathfrak{E}_n\left(\lambda:0\right)$
generated by \[\frac{\lambda}{e^{\lambda
t}}=\sum_{n\geq0}\mathfrak{E}_n\left(\lambda\right)\frac{t^n}{n!}\]
First we prove that
\[\mathfrak{E}_n\left(\lambda:x\right)=\lambda\left(x-\lambda\right)^n\]
and then
\[\mathfrak{E}_n\left(\lambda\right)=\left(-1\right)^n\lambda^{n+1}\]
These results go alone to show that some results in this paper are
trivial. For example the result in Theorem 2.10 p.387
\[D\mathfrak{E}_n\left(\lambda:x\right)=n\mathfrak{E}_{n-1}\left(\lambda:x\right)\]
and the result in Theorem 3.1 p.387
\[D_y\mathfrak{E}_n\left(\lambda:x+y\right)=D_x\mathfrak{E}_n\left(\lambda:x+y\right)=n\mathfrak{E}_{n-1}\left(\lambda:x+y\right).\]
Furthermore the Theorem 3.2 p.388
\[\int_{0}^{1}\mathfrak{E}_n\left(\lambda:x+y\right)dy=\frac{\mathfrak{E}_{n+1}\left(\lambda:x+1\right)-\mathfrak{E}_{n+1}\left(\lambda:y\right)}{n+1}\]
and then the Corollary 3.3 p.388
\[\int_{0}^{1}\mathfrak{E}_n\left(\lambda:x\right)dx=\frac{\mathfrak{E}_{n+1}\left(\lambda:1\right)-\mathfrak{E}_{n+1}\left(\lambda\right)}{n+1}.\]

\section{Some basic properties}
For any positive integer $n$, the polynomial
$\mathfrak{E}_n\left(\lambda:x\right)$  is a binomial polynomial
with weight $\lambda$, the following theorem states an improvement
of the expression (i) Theorem 2.2 \cite{Kang} p.384
\[\mathfrak{E}_n\left(\lambda:x\right)=\sum_{k=0}^{n}{n\choose
k}\mathfrak{E}_n\left(\lambda\right)x^{n-k}\]
\begin{theorem}\label{th1}
\begin{equation}\label{eqth1}
\mathfrak{E}_n\left(\lambda:x\right)=\sum_{k=0}^{n}{n\choose
k}\left(-1\right)^k\lambda^{k+1}x^{n-k}.
\end{equation}
Furthermore
\begin{equation}
\mathfrak{E}_n\left(\lambda\right)=\left(-1\right)^{n}\lambda^{n+1}
\end{equation}
\end{theorem}
\begin{proof}
Since
\[\frac{\lambda}{e^{\lambda
t}}e^{xt}=\lambda
e^{\left(x-\lambda\right)t}=\sum_{n\geq0}\lambda\left(x-\lambda\right)^n\frac{t^n}{n!}\]
then
\[\sum_{n\geq0}\mathfrak{E}_n\left(\lambda:x\right)\frac{t^n}{n!}=\sum_{n\geq0}\lambda\left(x-\lambda\right)^n\frac{t^n}{n!}.\]
 After comparison we deduce that
\[\mathfrak{E}_n\left(\lambda:x\right)=\lambda\left(x-\lambda\right)^n=\sum_{k=0}^{n}{n\choose k}\left(-1\right)^k\lambda^{k+1}x^{n-k}.\]
To get the second formula just remark that
\[\mathfrak{E}_n\left(\lambda:x\right)=\lambda\left(x-\lambda\right)^n=\left(-1\right)^{n}\lambda^{n+1}+\sum_{k=0}^{n-1}{n\choose
k}\left(-1\right)^k\lambda^{k+1}x^{n-k}\] and then for $x=0$ we
conclude that
\[\mathfrak{E}_n\left(\lambda\right)=\left(-1\right)^{n}\lambda^{n+1}\]
\end{proof}
The identity (ii) Theorem 2.2 \cite{Kang} p.384 is a consequence of
the Theorem \ref{th1}
\begin{corollary}\label{cor1}
\begin{equation}\label{eq1cor1}
\mathfrak{E}_n\left(\lambda:x+y\right)=\sum_{k=0}^{n}{n\choose
k}\mathfrak{E}_n\left(\lambda:x\right)y^{n-k}
\end{equation}
\end{corollary}
\begin{proof}
The identity \eqref{eq1cor1} Corollary \ref{cor1} follows from the
identity \eqref{eqth1} Theorem \ref{th1} as follows.
\[\mathfrak{E}_n\left(\lambda:x+y\right)=\sum_{k=0}^{n}{n\choose
k}\left(-1\right)^k\lambda^{k+1}\left(x+y\right)^{n-k}\]
\[\mathfrak{E}_n\left(\lambda:x+y\right)=\sum_{i=0}^{n}\sum_{k=0}^{i}{n\choose
k}{n-k\choose i-k}\left(-1\right)^k\lambda^{k+1}x^{i-k}y^{n-i}\] but
\[{n\choose
k}{n-k\choose
i-k}=\frac{n!i!}{k!\left(i-k\right)!\left(n-i\right)!i!}={n\choose
i}{i\choose k}\] then
\begin{equation}\label{eq2cor1}
\mathfrak{E}_n\left(\lambda:x+y\right)=\sum_{i=0}^{n}\sum_{k=0}^{i}{n\choose
i}{i\choose k}\left(-1\right)^k\lambda^{k+1}x^{i-k}y^{n-i}
\end{equation}
and the result \eqref{eq1cor1} Corollary \ref{cor1} follows.
\end{proof}
We attract attention that the identiy \eqref{eq2cor1} is an
improvement of the Theorem 3.4 \cite{Kang} p.389. Only in means of
the identity \eqref{eqth1} Theorem \ref{th1} a sample proof of the
identity in Theorem 2.4 \cite{Kang} p.385
\[x^n=\sum_{k=0}^{n}{n\choose
k}\lambda^{n-k-1}\mathfrak{E}_k\left(\lambda: x\right)\]  is just to
write
\[\sum_{k=0}^{n}{n\choose
k}\lambda^{n-k-1}\mathfrak{E}_k\left(\lambda:
x\right)=\sum_{k=0}^{n}{n\choose
k}\lambda^{n-k}\left(x-\lambda\right)^k=\left(x-\lambda+\lambda\right)^n=x^n.\]

Another proof of the identities (i) and (ii) in Theorem 2.3
\cite{Kang} is explained in the following theorem.
\begin{theorem}\label{th2}
\begin{equation}\label{eq1th2}
\mathfrak{E}_n\left(\lambda:
x\right)=\left(-1\right)^{n+1}\mathfrak{E}_n\left(-\lambda:
-x\right)
\end{equation}
\begin{equation}\label{eq2th2}
\mathfrak{E}_n\left(\lambda:
x\right)=2\mathfrak{E}_n\left(\frac{\lambda}{2}:
-\frac{\lambda}{2}+x\right)
\end{equation}
\end{theorem}
\begin{proof}
 Since we have $$\mathfrak{E}_n\left(-\lambda:
-x\right)=-\lambda\left(-x+\lambda\right)^n=-\left(-1\right)^n\lambda\left(x-\lambda\right)^n=\left(-1\right)^{n+1}\mathfrak{E}_n\left(\lambda:
x\right)$$ then \[\mathfrak{E}_n\left(\lambda:
x\right)=\left(-1\right)^{n+1}\mathfrak{E}_n\left(-\lambda:
-x\right).\] and \[\mathfrak{E}_n\left(\frac{\lambda}{2}:
-\frac{\lambda}{2}+x\right)=\frac{\lambda}{2}\left(-\frac{\lambda}{2}+x-\frac{\lambda}{2}\right)^n=\frac{\lambda}{2}\left(x-\lambda\right)^n=\frac{1}{2}\mathfrak{E}_n\left(\lambda:
x\right)\] and then \[\mathfrak{E}_n\left(\lambda:
x\right)=2\mathfrak{E}_n\left(\frac{\lambda}{2}:
-\frac{\lambda}{2}+x\right).\]
\end{proof}
A sample proof of the identity in Theorem 2.5 \cite{Kang} p.285
\begin{eqnarray*}
\sum_{k=0}^{n}{n\choose
k}\left(\lambda-x\right)^k\mathfrak{E}_{n-k}\left(\lambda: x\right)=
\left\{
\begin{array}{ccc}
\lambda\ &\quad \textrm{ if }\ n=0, \\
0\  &\quad  \textrm{ otherwise}
\end{array}
\right.
\end{eqnarray*}
is given as follows. It is trivial to see that for $n=0$ the sum is
$\lambda$ and if $n\geq1$ we have
\[\sum_{k=0}^{n}{n\choose
k}\left(\lambda-x\right)^k\mathfrak{E}_{n-k}\left(\lambda:
x\right)=\lambda\left(x-\lambda\right)^n\sum_{k=0}^{n}{n\choose
k}\left(-1\right)^k=0\] because \[\sum_{k=0}^{n}{n\choose
k}\left(-1\right)^k=\left(1-1\right)^n=0.\]
\begin{remark} For any cupel $\left(a,b\right)$ of numbers, the formulae in Theorem 2.8 \cite{Kang} p.386
\begin{eqnarray*}
\sum_{k=0}^{n}{n\choose
k}\left(\frac{a}{b}\right)^{n-2k}\mathfrak{E}_{n-k}\left(\frac{b\lambda}{a}:
\frac{bx}{a}\right)\mathfrak{E}_{k}\left(\frac{a\lambda}{b}:
\frac{ay}{b}\right)\\
\nonumber=\sum_{k=0}^{n}{n\choose
k}\left(\frac{b}{a}\right)^{n-2k}\mathfrak{E}_{n-k}\left(\frac{a\lambda}{b}:
\frac{ax}{b}\right)\mathfrak{E}_{k}\left(\frac{b\lambda}{a}:
\frac{by}{a}\right)
\end{eqnarray*}
results from the identities
\[\left(\frac{a}{b}\right)^{n-2k}\mathfrak{E}_{n-k}\left(\frac{b\lambda}{a}:
\frac{bx}{a}\right)\mathfrak{E}_{k}\left(\frac{a\lambda}{b}:
\frac{ay}{b}\right)=\left(x-\lambda\right)^{n-k}\left(y-\lambda\right)^{k}\]
and
\[\left(\frac{b}{a}\right)^{n-2k}\mathfrak{E}_{n-k}\left(\frac{a\lambda}{b}:
\frac{ax}{b}\right)\mathfrak{E}_{k}\left(\frac{b\lambda}{a}:
\frac{by}{a}\right)=\left(x-\lambda\right)^{k}\left(y-\lambda\right)^{n-k}\]
and the fact that \[\sum_{k=0}^{n}{n\choose
k}\left(x-\lambda\right)^{n-k}\left(y-\lambda\right)^{k}=\sum_{k=0}^{n}{n\choose
k}\left(x-\lambda\right)^{k}\left(y-\lambda\right)^{n-k}=\left(x+y-2\lambda\right)^n.\]

\end{remark}
In the case $x=y$ a new identity without the sum is obtained in the
following corollary.
\begin{corollary}\label{cor2}
\begin{equation}\label{eqcor2}
\left(\frac{a}{b}\right)^{2n-4k}\mathfrak{E}_{n-k}\left(\frac{b\lambda}{a}:
\frac{bx}{a}\right)\mathfrak{E}_{k}\left(\frac{a\lambda}{b}:
\frac{ax}{b}\right)=\mathfrak{E}_{n-k}\left(\frac{a\lambda}{b}:
\frac{ax}{b}\right)\mathfrak{E}_{k}\left(\frac{b\lambda}{a}:
\frac{bx}{a}\right)
\end{equation}
\end{corollary}
\begin{proof}
We have \[\mathfrak{E}_{n-k}\left(\frac{b\lambda}{a}:
\frac{bx}{a}\right)\mathfrak{E}_{k}\left(\frac{a\lambda}{b}:
\frac{ax}{b}\right)=\left(\frac{b}{a}\right)^{n-2k}\left(x-\lambda\right)^n\]
then
\[\left(\frac{a}{b}\right)^{2n-4k}\mathfrak{E}_{n-k}\left(\frac{b\lambda}{a}:
\frac{bx}{a}\right)\mathfrak{E}_{k}\left(\frac{a\lambda}{b}:
\frac{ax}{b}\right)=\left(\frac{a}{b}\right)^{n-2k}\left(x-\lambda\right)^n\]
with
\[\left(\frac{a}{b}\right)^{n-2k}\left(x-\lambda\right)^n=\mathfrak{E}_{n-k}\left(\frac{a\lambda}{b}:
\frac{ax}{b}\right)\mathfrak{E}_{k}\left(\frac{b\lambda}{a}:
\frac{bx}{a}\right)\] and the result follows.
\end{proof}
Finally for the cupel $\left(1,b\right)$, the identity
\eqref{eqcor2} Corollary \ref{cor2} becomes
\begin{corollary}\label{cor3}
\begin{equation}\label{eqcor3}
\left(\frac{1}{b}\right)^{2n-4k}\mathfrak{E}_{n-k}\left(b\lambda:
bx\right)\mathfrak{E}_{k}\left(\frac{\lambda}{b}:
\frac{x}{b}\right)=\mathfrak{E}_{n-k}\left(\frac{\lambda}{b}:
\frac{x}{b}\right)\mathfrak{E}_{k}\left(b\lambda: bx\right)
\end{equation}
\end{corollary}

\end{document}